\RequirePackage{fix-cm}
\documentclass[smallextended]{svjour3}       
\smartqed  
\usepackage{graphicx}
\usepackage[top=1in,bottom=1in,left=1in,right=1in]{geometry}
\usepackage{hyperref}
\usepackage{setspace}
\usepackage{pdfpages}
\usepackage{graphicx}
\usepackage{xcolor}
\usepackage{amssymb}
\usepackage{commath}
\usepackage{mathtools}

\DeclarePairedDelimiter\floor{\lfloor}{\rfloor}
\usepackage{algorithm}
\usepackage[noend]{algpseudocode}
\usepackage{amsfonts}
\newtheorem{observation}{Observation}
\newcommand{\change}[1]{{\color{black} #1}}
\newcommand{\rogers}[1]{{\color{black} #1}}

\begin{document}

\title{Dimension of CPT posets
}


\author{Atrayee Majumder         \and
        Rogers Mathew	\and
        Deepak Rajendraprasad 
}


\institute{Atrayee Majumder \at
              Department of Computer Science and Engineering,\\ Indian Institute of Technology Kharagpur,
Kharagpur, India - 721302.
              \email{atu.tua@gmail.com}           
           \and
          Rogers Mathew \at
              Department of Computer Science and Engineering,\\ Indian Institute of Technology Hyderabad, Telengana, India - 502285.
			\email{rogersmathew@gmail.com}
			\and
			Deepak Rajendraprasad \at
			Department of Computer Science and Engineering,\\ Indian Institute of Technology Palakkad,
Kerala, India - 678557.
			\email{deepakmail@gmail.com}
}

\date{Received: date / Accepted: date}

\maketitle

\begin{abstract}
A collection of linear orders on $X$, say $\mathcal{L}$, is said to \emph{realize} a partially ordered set (or poset) $\mathcal{P} = (X, \preceq)$ if, for any two distinct $x,y \in X$, $x \preceq y$ if and only if $x \prec_L y$, $\forall L \in \mathcal{L}$. We call $\mathcal{L}$ a \emph{realizer} of $\mathcal{P}$. The \emph{dimension} of $\mathcal{P}$, denoted by $dim(\mathcal{P})$, is the minimum cardinality of a realizer of $\mathcal{P}$. 
  
A \emph{containment model} $M_{\mathcal{P}}$ of a poset $\mathcal{P}=(X,\preceq)$ maps every $x \in X$ to a set $M_x$ such that, for every distinct $x,y \in X,\ x \preceq y$ if and only if $M_x \varsubsetneq M_y$. We shall be using the collection $(M_x)_{x \in X}$ to identify the containment model $M_{\mathcal{P}}$. A poset $\mathcal{P}=(X,\preceq)$ is a Containment order of Paths in a Tree (CPT poset), if it admits a containment model $M_{\mathcal{P}}=(P_x)_{x \in X}$ where every $P_x$ is a path of a tree $T$, which is called the host tree of the model.
 
We show that if a poset $\mathcal{P}$ admits a CPT model in a host tree $T$ of maximum degree $\Delta$ and radius $r$, then \rogers{$dim(\mathcal{P}) \leq \lg\lg \Delta + (\frac{1}{2} + o(1))\lg\lg\lg \Delta + 	\lg r + \frac{1}{2} \lg\lg r + 	\frac{1}{2}\lg \pi + 3$. This bound is asymptotically tight up to an additive factor of $\min(\frac{1}{2}\lg\lg\lg \Delta, \frac{1}{2}\lg\lg r)$. Further, let $\mathcal{P}(1,2;n)$ be the poset consisting of all the $1$-element and $2$-element subsets of $[n]$ under `containment' relation and let $dim(1,2;n)$ denote its dimension. The proof of our main theorem gives a simple algorithm to construct a realizer for $\mathcal{P}(1,2;n)$ whose cardinality is only an additive factor of at most $\frac{3}{2}$ away from the optimum.}  
\end{abstract}
\section{Introduction}
\subsection*{Dimension of a poset}
A partially ordered set or \emph{poset} $\mathcal{P}=(X,\preceq)$ is a tuple, where $X$ represents a set, and $\preceq$ is a binary relation on the elements of $X$ that is reflexive, anti-symmetric and transitive.
For any $x,y \in X,\ x$ is said to be \emph{comparable} with $y$ if either $x \preceq y$ or $y \preceq x$. Otherwise, we say $x$ and $y$ are \emph{incomparable}. 
A \emph{linear order} is a partial order in which every two elements are comparable with each other. If a partial order $\mathcal{P}=(X,\preceq)$ and a linear order $L=(X,\prec)$ are both defined on the same set $X$, and if every ordered pair in $\mathcal{P}$ are also present in $L$, then $L$ is called a \emph{linear extension} of $\mathcal{P}$.
A collection of linear orders, say $\mathcal{L}=\{L_1,\ L_2,\ \ldots,\ L_s\}$ with each $L_k$ defined on $X$, is said to \emph{realize} a poset $\mathcal{P}=(X,\preceq)$ if, for any two distinct elements $x_i, x_j \in X$, $x_i \preceq x_j \in \mathcal{P}$ if and only if $x_i \prec_{L_k} x_j,\ \forall L_k \in L$. We call $\mathcal{L}$ a \emph{realizer} for $\mathcal{P}$. The \emph{dimension of a poset} $\mathcal{P}$, denoted by $dim(\mathcal{P})$, is defined as the minimum cardinality of a realizer for $\mathcal{P}$. The concept of poset dimension was introduced by Dushnik and Miller in \cite{dushnik1941partially} and has been extensively studied since then (see \cite{trotter2001combinatorics}). Let $x,y \in X$ such that $x$ and $y$ are incomparable in $\mathcal{P}$. We say the ordered incomparable pair $(x,y)$ is \emph{critical} if 
\begin{enumerate}
\item $\forall u \in X\setminus\{x\}$, $u \preceq x \implies u \preceq y$, and 
\item $\forall v \in X\setminus\{y\}$, $y \preceq v \implies x \preceq v$. 
\end{enumerate}
Critical pairs were introduced by Rabinovitch and Rival \cite{rabinovitch1979rank}. The following theorem is from their paper. 
\begin{theorem}[\cite{rabinovitch1979rank}]\label{thm:rabinovitch}
A family $\mathcal{L}$ of linear extensions of a poset $\mathcal{P} = (X, \preceq)$ is a realizer of $\mathcal{P}$ if and only if, for every critical pair $(x,y) \in X \times X$, there is an $L \in \mathcal{L}$ with $y \prec_L x$. 
\end{theorem}

\subsection*{Containment model for representing a poset}
A \emph{containment model} $M_{\mathcal{P}}$ of a poset $\mathcal{P}=(X,\preceq)$ maps every $x \in X$ to a set $M_x$ such that, for every distinct $x,y \in X,\ x \preceq y$ if and only if $M_x \varsubsetneq M_y$. We shall be using the collection $(M_x)_{x \in X}$ to identify the containment model $M_{\mathcal{P}}$. The reader may note that, for any poset $\mathcal{P}=(X,\preceq)$, $M_x=\{y:\ y \preceq x\}$, $\forall x \in X$, is a valid containment model of $\mathcal{P}$. In  \cite{fishburn1985interval,Fishburn1998,golumbic1989containment,trotter2001combinatorics,urrutia1989partial} researchers have tried imposing geometric restrictions to the sets $M_x$ to obtain geometric containment models. To cite a few: Containment models in which $M_x$ is an interval on the $x$-axis \cite{dushnik1941partially,fishburn1985interval}, or every $M_x$ is a $d$-box in the $d$-Euclidean space \cite{golumbic1984containment,golumbic1989containment,urrutia1989partial}, or every $M_x$ is a $d$-sphere in the $d$-Euclidean space \cite{urrutia1989partial}.

\subsection*{Dimension of posets that admit a containment model}
It was shown by Dushnik and Miller in \cite{dushnik1941partially} that $dim(\mathcal{P}) \leq 2$ if the poset $\mathcal{P}$ admits an interval containment model. Golumbic \cite{golumbic1984containment} and Golumbic and Scheinerman \cite{golumbic1989containment} generalized this further, showing that $\mathcal{P}$ is a containment poset of axis-parallel $d$-dimensional boxes in $d$-dimensional Euclidean space if and only if $dim(\mathcal{P}) \leq 2d$. In \cite{Sidney1988} Sidney et al. stated that all posets of dimension $2$ admit a containment model named circle order where elements of the partial order are mapped to circles in the Euclidean plane.
Santoro and Urrutia showed in \cite{santoro1987angle} that every poset of dimension $3$ can be represented using a containment model where every element of the poset is mapped to an equilateral triangle in the Euclidean plane. They also showed that $dim(\mathcal{P}) \leq n$ when the poset $\mathcal{P}$ admits a containment model where the elements of $\mathcal{P}$ are represented by regular $n$-gons all having the same orientation in the Euclidean plane.
Trotter and Moore in \cite{TROTTER197679} studied the dimension of a poset that admits a containment model where every element of the poset is mapped to a subgraph of a given host graph. They proved the following interesting theorem.
\begin{theorem}
\label{thmintro1}
\cite{TROTTER197679}
If $G$ is a nontrivial connected graph with $n$ non-cut vertices,
 then the dimension of a poset $X(G)$ formed by the induced connected subgraphs of $G$ ordered by inclusion is $n$.
\end{theorem}

In this paper, we focus on \emph{Containment order of Paths in a Tree} (CPT), which was first introduced by Corneil and Golumbic 
\footnote{Corneil, D. and Golumbic, M. C., Unpublished, but cited in \cite{golumbic1989containment}, \cite{golumbic1984containment}},
and studied further by Alc{\'o}n et al. in \cite{ALCON2018139}, and Golumbic and Limouzy\cite{GoLi2012}. Below we define a CPT poset as outlined in \cite{ALCON2018139}.
\begin{definition}
A poset $\mathcal{P}=(X,\preceq)$ is a Containment order of Paths in a Tree (CPT poset), if there exists a tree $T$ such that $\mathcal{P}$ admits a containment model $M_{\mathcal{P}}=(P_x)_{x \in X}$ where every $P_x$ is a path of the tree $T$. $\mathcal{ T}$ will be called the host tree of the model.
\end{definition}
The following theorem stated in \cite{ALCON2018139} follows from \emph{Theorem~\ref{thmintro1}}.
\begin{theorem}
\label{thmintro2}
\cite{ALCON2018139}
If a poset $\mathcal{P}$ admits a CPT model in a host tree $T$ with $k$ leaves then $dim(\mathcal{P}) \leq k$.
\end{theorem}
The following observation follows directly from the definition of critical pairs. 
\begin{observation}\label{obv:criticalpair}
Let $\mathcal{P} = (X, \preceq)$ be a poset that admits a containment model $M_{\mathcal{P}} = (P_x)_{x \in X}$ in a host tree $T$, where each $P_x$ is a path in $T$. Then the critical pairs in $\mathcal{P}$ have the form $(x,y)$ where $P_x$ is a singleton vertex in $T$, say $v$, that is not in $P_y$, and there is no path extending $P_y$ that does not contain $v$ $(=P_x)$. 
\end{observation}


\subsection{Notations and definitions}
\label{A.2}
Unless mentioned explicitly, all logarithms used in the paper are to the base $2$. Given any $n \in \mathbb{N}$, we shall use $[n]$ to denote the set $\{1,2, \ldots, n\}$.

\rogers{
\begin{definition}
\label{def2}
Let $\mathcal{P}$ be a CPT poset that admits a containment model in a host tree $T$. Let $D$ be a planar drawing of $T$. A \emph{linear extension $L$ of $\mathcal{P}$ corresponding to a tree traversal $\lambda$ of $D$} is calculated according to the following rules: Retain all the poset relations. Let 
$last_i(\lambda)$ denote the vertex of a path $P_i$ in $T$ which was listed after every other vertex of $P_i$ was listed in $\lambda$. Let $P_i$ and $P_j$, $i < j$, be two paths representing elements $x_i$ and $x_j$, respectively, of $\mathcal{P}$. Suppose 
$last_i(\lambda) \neq last_j(\lambda)$. Then, $x_i \prec_{L} x_j$ if and only if 
$last_i(\lambda)$ was listed before 
$last_j(\lambda)$ in the tree traversal $\lambda$. Consider the case when  $last_i(\lambda)=last_j(\lambda)$. Then, $x_i \prec_{L} x_j$ if and only if $\abs{P_i} < \abs{P_j}$. 
\end{definition}
It is easy to observe that the linear order $L$ thus constructed is unique and is a linear extension of $\mathcal{P}$. 
} 
%


\section{A $3$-suitable family of permutations and $dim(1,2;n)$}
\begin{definition}
Let $S=\{R_1, R_2, \ldots, R_k\}$, where each $R_i$ is a permutation (or linear order or simple order) of $[n]$. We say $S$ is a $3$-suitable family of permutations of $[n]$ if for every $3$-subset of $[n]$, say $\{a_1, a_2, a_3\}$, and any distinguished element of the set, say $a_3$, there is some permutation $R_i\in S$ such that $a_j \prec_{R_i} a_3$ for every $1 \leq j < 3$, that is $a_3$ succeeds all the other elements of the $3$-subset under $R_i$. Let $\alpha(n)$ denote the cardinality of a smallest $3$-suitable family of permutations of $[n]$. 
\end{definition}

Let $\mathcal{P}(1,2;n)$ denote the poset representing $1$-element and $2$-element subsets of an $n$ element set ordered by the subset containment relation. Let $dim(1,2;n)$ denote the dimension of the poset $\mathcal{P}(1,2;n)$. Given a realizer $\mathcal{R}$ of the poset $\mathcal{P}(1,2;n)$, it is easy to construct a $3$-suitable family of permutations of $[n]$ of the same size from $\mathcal{R}$. In a similar way, given a $3$-suitable family of permutations of $[n]$, one can construct from it a realizer of $\mathcal{P}(1,2;n)$ of the same cardinality.  Thus, $dim(1,2;n) = \alpha(n)$. In this section we discuss some of the bounds on $dim(1,2;n)$. In \cite{HOSTEN1999133},  Ho{\c{s}}ten and Morris proved the following theorem.
\begin{theorem}[Theorem $1.1$ in \cite{HOSTEN1999133}]
\label{thm:Hosten}
Let $n \geq 3$. Then, $dim(1,2;n)$ is the smallest integer $t$ for which there are $n$ antichains in the subset lattice of $[t- 1]$ that do not contain $[t- 1]$ or two sets whose union is $[t - 1]$.
\end{theorem} 
The determination of the number of monotone Boolean functions on $n$ variables is known as Dedekind's problem. This problem is equivalent to calculating the number of antichains of the poset representing all possible subsets of an $n$ element set ordered by subset containment relation. Kleitman and Markovsky in \cite{10.2307/1998052} gave the following result regarding this problem.
\begin{theorem} [\cite{10.2307/1998052}]
\label{thm:Kleitman}
The size of the free distributive lattice, i.e. $\Psi(n)$, on $n$ generators (which is the number of monotone Boolean functions on $n$ variables or the number of antichains in the subset lattice of $[n]$), satisfies the following condition.
$$\Psi(n) \leq 2^{(1+O(\log{n}/n)) {n \choose {\floor{n/2}}}}$$
\end{theorem} 
Combining \emph{Theorem~\ref{thm:Hosten}} and \emph{Theorem~\ref{thm:Kleitman}}, Agnarsson, Felsner, and Trotter has given the following result in \cite{AGNARSSON19995}.
\begin{theorem} [\cite{AGNARSSON19995}] \label{AgnFelsTrot} For any $n \geq 3$, $dim(1,2;n) \leq \lg{\lg{n}} + (1/2 +o(1))\lg{\lg{\lg{n}}}.$ 
\end{theorem} 
Following is a more technical result due to Trotter which combines Theorem  \ref{thm:Hosten} and Theorem \ref{thm:Kleitman}.
\begin{theorem}
\label{trotterthm}
For every $\epsilon > 0$, there is an integer $n_0$ so that if $n > n_0$ and $$s = \lg{\lg{n}} + \frac{1}{2}\cdot\lg{\lg{\lg{n}}} + \frac{1}{2} \cdot \lg{\pi} + \frac{1}{2},\ then$$ $$s - \epsilon < dim(1,2;n) < s + 1 + \epsilon.$$
\end{theorem} 
\section{The dimension of a CPT poset}

\change{

Given a tree $T$, let $\mathcal{P}(T)$ denote the containment poset of all paths
in $T$ and let $dim_p(T)$ denote $dim(\mathcal{P}(T))$.

An antichain $A$ is called {\em intersecting}, if no pair of sets in $A$ is
disjoint.  Let $\beta(t)$ denote the smallest natural number $\beta$ such that
the subset lattice of $[\beta]$ contains an intersecting antichain of size $t$.
It suffices to ensure that \rogers{${\beta \choose \lceil (\beta+1)/2 \rceil} \geq
t$. Hence $\beta(t) \leq \left\lceil\lg t + \frac{1}{2} \lg\lg t + \frac{1}{2}\lg \pi + 1\right\rceil$.}

\begin{theorem} \label{thm:fullposet}
Let $T$ be a rooted tree of height $h$ in which every internal vertex has
exactly $k$ children (a {\em perfect $k$-ary tree} in Computer Science terms).
Then $dim_p(T) \leq \alpha(k) + \beta(h) + 1$.
\end{theorem}

\begin{proof} Let $\alpha = \alpha(k)$ and $\beta = \beta(h)$. Fix a planar
drawing $D$ of $T$ with every child being below its parent. For each internal
node of $T$, identify its $k$ children with $[k]$ in the left to right order
they appear in $D$. Let $\psi = \{\sigma_1, \ldots, \sigma_{\alpha}\}$ be a
$3$-suitable family of linear orders of $[k]$. For each $i \in [\alpha]$, let
$D_i$ be the drawing of $T$ obtained by reordering the $k$ children of every
internal node in $D$ according to $\sigma_i$. Let $\phi = \{A_0, \ldots,
A_{h-1}\}$ be an antichain in the subset lattice of $[\beta]$. For each $i \in
\{0, \ldots, h-1\}$, associate to every vertex $v$ in the $i$-th level of $T$
the set $S(v) = A_i$.  For each $j \in [\beta]$, let $D_{\alpha + j}$ be the
drawing obtained from $D_1$ by reversing the order of the children of every
node $v$ where $j \in S(v)$. For each $i \in [\alpha + \beta]$, let $L_i$ be
the linear extension of $\mathcal{P}(T)$ corresponding (Definition~\ref{def2})
to the preorder traversal of the drawing $D_i$.  Finally let $D_0$ be the
linear extension of $\mathcal{P}(T)$ corresponding to any order of $V(T)$ in
which every parent appears only after all its children (a bottom to top
traversal).  We argue that $\mathcal{L} = \{L_0, \ldots, L_{\alpha + \beta} \}$
is a realizer of $\mathcal{P}(T)$ and thus the theorem.

From Theorem \ref{thm:rabinovitch}, we know that, in order to show that
$\mathcal{L}$ is a realizer for a poset $\mathcal{P}$, it is enough to show
that $\mathcal{L}$ is a collection of linear extensions of $\mathcal{P}$ such
that for every critical pair $(x,y)$ of $\mathcal{P}$, there is an $L \in
\mathcal{L}$ satisfying $y \prec_L x$. Consider a critical pair $(x, y)$ in
$\mathcal{P}(T)$. It follows from Observation~\ref{obv:criticalpair} and the
fact that $\mathcal{P}(T)$ contains every path in $T$, that $x$ is a vertex of
$T$ (call it $p$) and $y$ is a path in $T$ in which one end (call it $q_1$) is
a leaf of $T$ and the other (call it $q_2$) is either a leaf of $T$ or a
degree-$2$ neighbour of $p$. We will call the least common ancestor of $q_1$
and $q_2$ as $q$. For every path in a tree, the last vertex to be traversed in
a preorder traversal will always be an end-vertex of the path.  This is because
every node in a path is an ancestor to at least one end-vertex of the path.
Hence if $p$ succeeds both $q_1$ and $q_2$ in some preorder traversal, then we
have $y \prec_{L} x$, in the corresponding linear extension.

First let us consider the case when $q_2$ is a neighbour of $p$. If $p$ is
the parent of $q_2$, then $p$ succeeds both $q_1$ and $q_2$ in $L_0$. If
$p$ is a child of $q_2$ and then $p$ succeeds both $q_1$ and $q_2$ either in
$L_1$ or in $L_{\alpha + j}$ for every $j \in S(q)$.

Now we consider the case when $q_2$ (along with $q_1$) is also a leaf of $T$.
With relabeling if necessary, we can assume that $q_2$ is to the right of $q_1$
in $D_1$ (i.e., $q_1 \prec q_2$ in $L_1$).  If $p$ is an ancestor of $q$, then
$p$ succeeds both $q_1$ and $q_2$ in $L_0$.  If $p$ is neither an ancestor nor
a descendent of $q$, then $p$ succeeds both $q_1$ and $q_2$ either in $L_1$ or
in $L_{\alpha + j}$ for every $j \in S(r)$, where $r$ is the least common
ancestor of $p$ and $q$. A careful observation will reveal that all the cases
considered so far could have been handled with just three linear extensions -
$L_0$, $L_1$ and a linear extension corresponding to the right to left preorder
traversal of $D_1$. The only remaining case of $p$ being a descendent of $q$ is
the most demanding and we are forced to split it further into the following
three subcases. 

{\em Subcase 1}. Let $q$ be the only ancestor of $p$ in the path from $q_1$ 
to $q_2$. Let $q_1'$, $q_2'$ and $p'$ be, respectively, the ancestors of
$q_1$, $q_2$ and $p$ which are children of $q$. Let $D_i$, $i \in [\alpha]$
be a drawing in which $p'$ succeeds both $q_1'$ and $q_2'$. Such a drawing
will exist since $\psi$ is $3$-suitable for $[k]$. One can verify that $p$
succeeds both $q_1$ and $q_2$ in $L_i$.

{\em Subcase 2.} Let the least common ancestor $p_2$ of $p$ and $q_2$ be a
proper descendent of $q$. If $p$ is to the right of $q_2$ in $D_1$, then $p$
succeeds both $q_1$ and $q_2$ in $L_1$. Otherwise $p$ succeeds both $q_1$ and
$q_2$ in $L_{\alpha + j}$, for every $j \in S(p_2) \setminus S(q)$, which is
non-empty since $\phi$ is an antichain.

{\em Subcase 3.} Let the least common ancestor $p_1$ of $p$ and $q_1$ be a
proper descendent of $q$. If $p$ is to the right of $q_2$ in $D_1$, then $p$
succeeds both $q_1$ and $q_2$ in $L_{\alpha + j}$, for every $j \in S(q)
\setminus S(p_1)$, which is non-empty since $\phi$ is an antichain.  .
Otherwise $p$ succeeds both $q_1$ and $q_2$ in $L_{\alpha + j}$, for every $j
\in S(p_1) \cap S(q)$, which is non-empty since $\phi$ is intersecting.
\qed
\end{proof}

\begin{corollary}\label{cor:expandedbound}
For the tree $T$ in Theorem~\ref{thm:fullposet}, 
\rogers{$dim_p(T) \leq 
	\lg\lg k + (\frac{1}{2} + o(1))\lg\lg\lg k + 
	\lg h + \frac{1}{2} \lg\lg h + 
	\frac{1}{2}\lg \pi + 3$.}
\end{corollary}

\begin{remark}
The tree $T$ in Theorem~\ref{thm:fullposet} has $k^h$ leaves. The subposet of
$\mathcal{P}(T)$ induced on the singleton paths corresponding to every leaf of
$T$ and the maximal paths in $T$ is isomorphic to $\mathcal{P}(1,2;k^h)$. Hence
$dim_p(T) \geq dim(1,2;k^h) = \alpha(k^h)$. \rogers{Using the bound given in Theorem \ref{trotterthm} for $dim(1,2;n)$, we can say that the upper bound in
Corollary~\ref{cor:expandedbound} is tight up to an additive factor of
$\min(\frac{1}{2}\lg\lg\lg k, \frac{1}{2}\lg\lg h) + \frac{1}{2} \lg \pi + 4$. Hence our bound in Corollary \ref{cor:expandedbound} is asymptotically tight when at least one of $k$ or $h$ is a constant.}
\end{remark}

\begin{remark}
For a perfect binary tree $T$ ($k = 2$) of height $h$, we get 
$dim(1,2; 2^h) \leq dim_p(T) \leq \beta(h) + 2$.
\end{remark}
\rogers{
While considering only singleton paths corresponding to every leaf of $T$ and the maximal paths in $T$, the linear order $L_0$ defined in the proof of Theorem \ref{thm:fullposet} is not required as we do not encounter the case where $p$ is an ancestor of $q$ (or $q_2$). This observation helps us in arriving at Remark \ref{rem:3-suitable} below.
Let $\gamma(t)$ denote the smallest natural number $\gamma$ such that the subset lattice of $[\gamma]$ contains an intersecting antichain of size $t$ that satisfies the following property: for every two sets in the antichain, their union is a proper subset of $[\gamma]$.

\begin{remark}
\label{rem:3-suitable}
For any natural number $n \geq 3$, $dim(1,2; n) \leq \gamma(\lg n) \leq \beta(\lg n)+ 1 \leq \\ \left\lceil\lg \lg n + \frac{1}{2} \lg\lg \lg n + \frac{1}{2}\lg \pi + 1\right\rceil + 1$.
\end{remark}
Observe that this bound for $dim(1,2;n)$ is only an additive factor of at most $\frac{3}{2}$ away from the upper bound given by Theorem \ref{trotterthm}. Our proof yields a simple algorithm to construct a realizer for $\mathcal{P}(1,2;n)$ (or a $3$-suitable family of permutations of $[n]$) which is near optimal in size. 
  
}
\begin{corollary}
\label{cor:genposet}
If a poset $\mathcal{P}=(X, \preceq)$ admits a CPT model in a host tree $T$ of maximum degree $\Delta$ and radius $r$, then \rogers{$dim(\mathcal{P}) \leq \alpha(\Delta) + \beta(r) + 1$}.
\end{corollary}
\begin{proof}
$T$ is an induced subgraph of a perfect $\Delta$-ary tree $T$ of
height $r$ and $\mathcal{P}$ is hence an induced subposet of $\mathcal{P}(T)$.
\qed
\end{proof}

}

\begin{acknowledgements}
The authors are grateful to Prof. William T. Trotter for his detailed review comments that helped improve this manuscript significantly. We also thank Prof. Martin Golumbic and Prof. Vincent Limouzy for fruitful discussions on the topic of CPT graphs and posets.
\end{acknowledgements}
\bibliographystyle{acm}
\bibliography{mybibfile}

%

\end{document}